\documentclass{amsproc}
\usepackage{mathabx}
\usepackage{mathtools}

\usepackage{}
\copyrightinfo{2009}{American Mathematical Society}

\newtheorem{theorem}{Theorem}[section]
\newtheorem{lemma}[theorem]{Lemma}
\newtheorem{proposition}[theorem]{Proposition}
\newtheorem{corollary}[theorem]{Corollary}

\usepackage{tikz}
\usepackage{calc}
\usetikzlibrary{decorations.markings}
\tikzstyle{vertex}=[circle, draw, inner sep=0pt, minimum size=1pt]

\theoremstyle{definition}
\newtheorem{definition}[theorem]{Definition}
\newtheorem{example}[theorem]{Example}

\theoremstyle{remark}

\numberwithin{equation}{section}

\DeclarePairedDelimiter\ceil{\lceil}{\rceil}

\begin{document}

\title{Pancyclic zero divisor graph over the ring $\mathbb{Z}_n[i]$}

\author{Ravindra Kumar}
\address{Department of Mathematics,IIT Patna, Bihta campus, Bihta-801 106}
\curraddr{}
\email{ravindra.pma15@iitp.ac.in}
\thanks{}

\author{Om Prakash}
\address{Department of Mathematics \\
IIT Patna, Bihta campus, Bihta-801 106}
\curraddr{}
\email{om@iitp.ac.in}
\thanks{}



\date{}

\maketitle
\begin{abstract}
Let $\Gamma(\mathbb{Z}_n[i])$ be the zero divisor graph over the ring $\mathbb{Z}_n[i]$. In this article, we study pancyclic properties of $\Gamma(\mathbb{Z}_n[i])$ and $\overline{\Gamma(\mathbb{Z}_n[i])}$ for different $n$. Also, we prove some results in which $L(\Gamma(\mathbb{Z}_n[i]))$ and $\overline{L(\Gamma(\mathbb{Z}_n[i]))}$ to be pancyclic for different values of $n$.

\end{abstract}

\section{INTRODUCTION}

Let $R$ be a finite commutative ring with unity, $Z(R)$ the set of zero-divisors of $R$ and $Z^*(R) = Z(R)- \{0\}$. The zero-divisor graph of $R$, denoted by $\Gamma(R)$, is the graph in which the set of vertices $V(\Gamma(R))$ is $Z^*(R)$ and any two vertices $x, y \in V(\Gamma(R))$ are adjacent if and only if $xy = 0$.\\
It is known that the set of complex numbers forms a Euclidean domain under usual addition and multiplication of complex numbers where Euclidean norm is defined as $\lvert a+ib \rvert = a^2+b^2$. The set of Gaussian integers $\mathbb{Z}[i]$ is a subset of $\mathbb{C}$ which is defined as $\mathbb{Z}[i] = \{\alpha = a+ib \mid ~ a,b \in \mathbb{Z}\}$ and Gaussian norm $N(\alpha) = \alpha \overline{\alpha}$. It is obvious that a Gaussian integer is prime in $\mathbb{Z}[i]$ if its norm is prime in $\mathbb{Z}$. So the Gaussian prime can describe as follows:
\begin{enumerate}
	\item{$1+i$ and $1-i$ are Gaussian primes. }
	
	\item{If $q$ is a prime integer such that $q \equiv 3~(mod~4)$, then $q$ is a Gaussian prime. }
	
	\item{If $p= a^2+b^2$ is a prime for some integers $a$ and $b$ such that $p \equiv 1~(mod~4)$, then $a+ib$, $a-ib$ are Gaussian primes.}\\
\end{enumerate}
Let $n$ be a positive integer and $\langle n\rangle$ be the principal ideal generated by $n$ in $\mathbb{Z}[i]$. Then $\mathbb{Z}[i]/\langle n\rangle \cong \mathbb{Z}_n[i]$ and if $n = \prod_{k=1}^{m} t_{k}^{n_k}$, then $\mathbb{Z}_n[i] \cong \prod_{k=1}^{m} \mathbb{Z}_{t_k^{n_k}}[i]$, for detail reader can see \cite{factor}.

In $2008$, Osba et al. \cite{osba} introduced the zero divisor graph for the ring of Gaussian integers modulo $n$, where they discussed several graph theoretic properties for $\Gamma(\mathbb{Z}_n[i])$. \\
Through out the article, $p$ and $q$ represent the primes which are congruent to $1$ modulo $4$ and congruent to $3$ modulo $4$ respectively. For a connected graph $G$, the distance $d(u, v)$ is the shortest path between $u$ and $v$. A graph $G$ of order $n$ is said to be Hamiltonian if it contains a cycle of length $n$. The line graph of $G$, denoted by $L(G)$, is a graph whose vertices are the edges of $G$ and two vertices of $L(G)$ are adjacent whenever the corresponding edges of $G$ are adjacent. For basic definitions and results, we refer \cite{bala}.\\
A graph $G$ of order $p \geq 3$ is said to be pancyclic if $G$ contains a cycle of length $n$ for every integer $n$ where $3 \leq n \leq p$. If a graph contains every cycle of even length $n$, where $4 \leq n \leq p$, then the graph is said to be bipancyclic.

\section{When is $\Gamma(\mathbb{Z}_n)$ and $\Gamma(\mathbb{Z}_n[i])$ is Pancyclic.}

In this section we discuss the cases in which the graph $\Gamma(\mathbb{Z}_n)$ and $\Gamma(\mathbb{Z}_n[i])$ are pancyclic.

\begin{theorem}
 $\Gamma(\mathbb{Z}_n[i])$ is not pancyclic for $n=2^m, m>1$.

\end{theorem}

\begin{proof}

For $n=2^m$ where $m>1$, Graph $\Gamma(\mathbb{Z}_n[i])$ contains $2^{2m-2}$ pendent vertices. Since by Theorem $3$ of {\cite{osba}}, $\lvert \Gamma(\mathbb{Z}_{2^m}[i]) \lvert~ = 2^{2m-1}-1$ but it can not contain a cycle of length $2^{2m-1}-1$. Hence, $\Gamma(\mathbb{Z}_{2^m}[i])$ is not a pancyclic.
\end{proof}

\begin{theorem}

The graph $\Gamma(\mathbb{Z}_{q^m}[i])$ is pancyclic if and only if $m=2$.
\end{theorem}

\begin{proof}

Since $\mathbb{Z}_q[i]$ is a field, therefore $\Gamma(\mathbb{Z}_q[i])$ is an empty graph. For $m=2$, $\Gamma(\mathbb{Z}_{q^2}[i])$ is the complete graph $K_{q^2-1}$ and complete graph is always pancyclic. Now, for $m>2$, from Theorem 8 of \cite{osba1}, $\Gamma(\mathbb{Z}_{q^m}[i])$ is not Hamiltonian. Therefore, it does not contain a cycle of length $q^{2m-2}-1$. Hence, $\Gamma(\mathbb{Z}_{q^m}[i])$ is not pancyclic.\\
\end{proof}

\begin{theorem}

The graph $\Gamma(\mathbb{Z}_{p^m}[i])$ is bipancyclic if and only if $m=1$.

\end{theorem}

\begin{proof}

 $\Gamma(\mathbb{Z}_{p}[i])$ is the complete bipartite graph $k_{p-1,p-1}$ with two sets of vertices  $V_1 = <a+ib>- ~\{ \overline{0}\}$ and $V_2 = <a-ib>- ~\{ \overline{0}\}$. So it is bipancyclic graph. Now, for $m>1$, we know from Theorem $6$ of \cite{osba1} that $\Gamma(\mathbb{Z}_{p}[i])$ is not a Hamiltonian graph. Therefore, it is not a pancyclic with $ \lvert \Gamma(\mathbb{Z}_{p^m}[i]) \rvert = 2p^{2m-1}-p^{2m-2}-1$, an even integer. Hence, it is not a bipancyclic graph.

\end{proof}

It is known that for two primes $q_1$ and $q_2$ such that $q_j \equiv 3(mod4)$ for $j= 1,2$, $\Gamma(\mathbb{Z}_{q_1q_2}[i]) \cong \Gamma(\mathbb{Z}_{q_1}[i] \times \mathbb{Z}_{q_2}[i])$. Also, $\Gamma(\mathbb{Z}_{q_j}[i])$ is a field, therefore $\Gamma(\mathbb{Z}_{q_1q_2}[i]) \cong K_{{q_1}^2-1,{q_2}^2-1 }$. Hence, $\Gamma(\mathbb{Z}_{q_1q_2}[i])$ is a complete bipartite graph with $\lvert \Gamma(\mathbb{Z}_{q_1q_2}[i]) \rvert = q_1^2 + q_2^2 -2$. Since there does not exist any cycle of length $q_1^2 + q_2^2 -2$, thus, $\Gamma(\mathbb{Z}_{q_1q_2}[i])$ is neither pancyclic nor bipancyclic.\\

Now, we present the well-known existence theorem (Theorem 6.3.4 of \cite{bala}) on Hamiltonian graph:
\begin{proposition}
	If $G$ is a Hamiltonian, then for every nonempty proper subset $S$ of $V(G)$, $c(G - S) \leq \lvert S\rvert$.
\end{proposition}

\begin{lemma}
	For $n = p_1.p_2...p_n$ and $p_1 < p_2 <...<p_n$ are distinct primes, $\Gamma(\mathbb{Z}_n)$ is not Hamiltonian.
\end{lemma}

\begin{proof}
	
	Suppose $n= p_1.p_2...p_n$ where $p_{i}'s$ are distinct primes. We know that the set of vertices of $\Gamma(\mathbb{Z}_n)$ is all zero divisors of $\mathbb{Z}_n$. Let $S = \{\alpha : \alpha \in \{ (p_2.p_3...p_n), 2(p_2.p_3...p_n),..., (p_1-1)(p_2.p_3...p_n) \}\}$ and $H = \{ p_1, 2p_1,..., (p_2-1)p_1\}, ~ p_1 < p_2$. Then $c(\Gamma(\mathbb{Z}_n) - S) > \lvert H\rvert = p_2-1 > p_1 - 1 = \lvert S\rvert$. Follows from Proposition $2.4$, $\Gamma(\mathbb{Z}_n)$ is not a Hamiltonian.
	
\end{proof}

From the Lemma $2.5$, we can easily see that $\Gamma(\mathbb{Z}_n)$ is not a pancyclic graph. Now, we will show that $\Gamma(\mathbb{Z}_{p^m})$ is pancyclic if and only if $m = 2$.

\begin{lemma}
	$\Gamma(\mathbb{Z}_{p^m})$ is Hamiltonian if and only if $m = 2$ and $p$ is a prime.
\end{lemma}

\begin{proof}
	
	Since for $m = 1$, $\mathbb{Z}_p$ is a field so $\Gamma(\mathbb{Z}_p)$ is a null graph. For $m = 2$, $\Gamma(\mathbb{Z}_{p^2})$ is a complete graph of $p-1$ vertices, which is a Hamiltonian graph. Now, for $m > 2$, the vertex set in $\Gamma(\mathbb{Z}_{p^m})$ is $ \langle p\rangle - \{0\}$. Let $S = \{\alpha p^{m-1} : 1 \leq \alpha \leq p-1\}$ and $H = \{\alpha p \in \mathbb{Z}_{p^m}: gcd(\alpha, p) = 1\}$. Here, elements of $H$ are only adjacent to elements of $S$. Then $c(\Gamma(\mathbb{Z}_{p^m}) - S) > \lvert H\rvert > p-1  = \lvert S\rvert$. Hence, $\Gamma(\mathbb{Z}_{p^m})$ is not Hamiltonian.
	
\end{proof}

\begin{theorem}

$\Gamma(\mathbb{Z}_{p^m})$ is pancyclic if and only if $m=2$.

\end{theorem}

\begin{theorem}
	
	$\Gamma(\mathbb{Z}_{p^2q^2})$ is not a Hamiltonian graph for all distinct prime $p$ and $q$ with $p<q$.
	
\end{theorem}

\begin{proof}
	
	Let $S = \langle pq^2\rangle$ and $H = \{ \langle p\rangle - \{\langle p^2\rangle, \langle pq\rangle\}\}$. Then $H \subseteq V(\Gamma(\mathbb{Z}_{p^2q^2}))$ $ -S $. Now, $c(\Gamma(\mathbb{Z}_{p^2q^2})- S) > \lvert H\rvert = q(p-1)(q-1) >p-1 = \lvert S\rvert$. So, it follows from proposition $2.4$, $\Gamma(\mathbb{Z}_{p^2q^2})$ is not Hamiltonian.
	
\end{proof}

\begin{theorem}
	
	Let $R_1$ and $R_2$ be two rings and $R= R_1 \times R_2$. Then $\Gamma(R)$ is bipancyclic if and only if $R_1$ and $R_2$ are integral domains such that $\lvert R_1\rvert = m =\lvert R_2\rvert $.
	
\end{theorem}

\begin{proof}
	
	Suppose $R= R_1 \times R_2$, where $R_1$ and $R_2$ are integral domains and $\lvert R_1\rvert = \lvert R_2\rvert = m $, then $\Gamma(R)$ is a complete bipartite graph with two vertex sets $A_1 = \{ (x, 0) : x \in R_1\backslash \{0\}\}$ and $A_2 = \{ (0, y) : y \in R_2\backslash \{0\}\}$. So $\Gamma(R)$ is a bipancyclic graph. \\
Conversely, let $\Gamma(R)$ is bipancyclic. If possible, let $R_1$ be not an integral domain. Then there arises two cases:
	
	\begin{enumerate}
		\item{Let $\lvert Z(R_1)^*\rvert = 2k$ and $k \in \mathbb{Z}^*$. Then number of vertices in $\Gamma(R)$ is always even. In order to prove $\Gamma(R)$ is not Hamiltonian, consider the set $S =\{ (x,0): x \in R_1\}$ such that $\lvert S\rvert = m-1$. Then $c(\Gamma(R)- S) > m > m-1 = \lvert S \rvert$. By proposition $2.4$, $\Gamma(R)$ is not Hamiltonian, i.e. it does not contain a cycle of even length of $\lvert \Gamma(R)\rvert$. Hence, $\Gamma(R)$ is not bipancyclic.}
		
		\item{Suppose $\lvert Z(R_1)^*\rvert = 2k+1$ $k \in \mathbb{Z}^*$. Here, $\lvert \Gamma(R)\rvert = (2k+3)(m-1)$ and by proposition $2.4$, it is clear that $\Gamma(R)$ is not a Hamiltonian. If $m$ is an odd, then $\Gamma(R)$ is not bipancyclic and if $m$ is even, then order of  $\Gamma(R)$ is odd. Remove one vertex from $\Gamma(R)$ so that order of $\Gamma(R)$ is even. Then by proposition $2.4$, it is clear that $\Gamma(R)$ is not Hamiltonian. Hence, $\Gamma(R)$ is not bipancyclic.}
	\end{enumerate}
	Again, suppose $R_1$ and $R_2$ are not integral domains, then $Z(R_1)= \{x_1, x_2,...,x_r\}, Reg(R_1) = \{u_1, u_2,...,u_s\}$, $Z(R_2)= \{y_1, y_2,...,y_t\}$ and $Reg(R_2) = \{v_1, v_2,...,v_w\}$ such that $r+s= m= t+w$. Hence, by Proposition $2.4$, $\Gamma(R)$ is not bipancyclic.

	\end{proof}

\begin{example}
	
	Take $R_1 = \mathbb{Z}_3$ and $ R_2 = \mathbb{Z}_5 $. Then $ R = \mathbb{Z}_3 \times \mathbb{Z}_5 $ and $ Z(R)^* = (0,1), (0,2), (0,3), (0,4), (1,0), (2,0) $. We can see that $ \Gamma(R) $ is a complete bipartite graph of order $6$ i.e $ K_{2,4} $ but there is a no cycle in $ \Gamma(R) $ of order $6$.
	
\end{example}

\section{When is $L(\Gamma(\mathbb Z_n))$ and $L(\Gamma(\mathbb Z_n[i]))$ pancyclic$?$}

For a commutative ring $R$, it is clear that $\Gamma(R)$ is connected by (Theorem 2.3 of \cite{ander}), so $L(\Gamma(R))$ is also connected. To characterise the graph $L(\Gamma(\mathbb Z_n))$ and $L(\Gamma(\mathbb Z_n[i]))$ is  pancyclic we use the following proposition.

\begin{proposition}

 (Corollary 5 of \cite{ben}) Let $G$ be a connected, almost bridgeless graph of order $n \geq 4$ such that $deg(u) +deg(v) \geq \frac{(2n+1)}{3}$ for every edge $uv$ of $G$. Then $L(G)$ is Hamiltonian . Moreover if $G \ncong C_4, C_5$, then $L(G)$ is pancyclic.

\end{proposition}

\begin{corollary}

 (Corollary 8 of \cite{veld}) If $G$ is a graph of diameter at most $2$ with $\lvert V(G) \rvert \geq 4$, then $L(G)$ is Hamiltonian.

\end{corollary}

\begin{theorem}
$L(\Gamma(\mathbb{Z}_{pq}))$ is pancyclic graph for $p < q$, where $p$ and $q$ are distinct primes. \\

\end{theorem}

\begin{proof}

 Suppose $p = 2$ and $q > 3$, then $\Gamma(\mathbb{Z}_{2q})$ is a star graph and its line graph $L(\Gamma(\mathbb{Z}_{2q}))$ is a complete graph of order $q-1$. So $L(\Gamma(\mathbb{Z}_{2q}))$ is pancyclic. Now, if $p, q$ are distinct odd primes, then $\Gamma(\mathbb{Z}_{pq})$ is a complete bipartite graph $K_{p-1, q-1}$ and $diam(\Gamma(\mathbb{Z}_{pq})) \leq 2$. Hence, by Corollary $3.2$, $L(\Gamma(\mathbb{Z}_{pq}))$ is a Hamiltonian graph. Now, for any $uv \in E(\Gamma(\mathbb{Z}_{pq}))$, $d(u) + d(v) = n > \frac{(2n+1)}{3}$ where $n = \lvert V(G) \rvert$ and $G \ncong C_4, C_5$. Thus, by Proposition $3.1$, $L(G)$ is pancyclic.

\end{proof}

\begin{theorem}

$L(\Gamma(\mathbb{Z}_{p^m}))$ is a pancyclic for a prime $p >3$ and $ m = 2, 3$.

\end{theorem}

\begin{proof}
 For $m = 2$, $\Gamma(\mathbb{Z}_{p^2})$ is a complete graph of order $p-1$ and complete graph is always bridgeless such that $d(u) + d(v) = 2p-4 > \frac{(2n+1)}{3}$, where $n = \lvert \Gamma(\mathbb{Z}_{p^2}) \rvert$ and $uv \in E(\Gamma(\mathbb{Z}_{p^2}))$. Now, for $m = 3$, $\Gamma(\mathbb{Z}_{p^3})$ is also bridgeless graph such that $d(u) + d(v) \geq p^2+p-3 > \frac{(2n+1)}{3}$, where $n = \lvert \Gamma(\mathbb{Z}_{p^3}) \rvert = p^2-1$. Hence, for both cases, $d(u) + d(v) \geq \frac{(2n+1)}{3}$ for every $uv \in E(\Gamma(\mathbb{Z}_{p^m}))$. Therefore, from proposition $3.1$, $L(G)$ is Hamiltonian and $G \ncong C_4, C_5$. Thus, $L(G)$ is a pancyclic graph.

\end{proof}

\begin{definition}
	
	Let $G$ be a graph of order $p \geq 5$. Then $G$ is said to be an $R$-graph if there exist distinct $r, s, t, u \in V(G)$ such that $rs, st, tu, ur \in E(G)$ and for every $v \in V(G-r-s-t-u)$, either $rv \in E(G)$ or $tv \in E(G)$.
	
\end{definition}

\begin{lemma}
	
	\cite{pancyclic} Let $G$ be a graph of order $p \geq 5$. If $G$ is an $R$- graph, then $L(G)$ is pancyclic.
	
\end{lemma}

\begin{theorem}
	
	 For integer $m \geq 2$, $L(\Gamma(\mathbb{Z}_{2^m}[i]))$ is pancyclic.
	
\end{theorem}

\begin{proof}
	
	For $m = 2$, $L(\Gamma(\mathbb{Z}_{2^2}[i]))$ is a graph of order $7$. Then it has induced complete subgraph of order $6$ and a vertex of degree $2$. Therefore, the graph contain all cycles of length $k$, for $3 \leq k \leq 7$. Hence, $L(\Gamma(\mathbb{Z}_{2^2}[i]))$ is pancyclic. Now, for $m >2$, consider distinct four vertices $r, s, t, u $ of $V(\Gamma(\mathbb{Z}_{2^m}[i]))$ where $r = 2^{n-1} +i2^{n-1}$, $s = 2^{n-1}$, $t =2$, $u = i2^{n-1}$ such that $rs, st, tu, ur$ have an edge. Also, for every $v \in V(\Gamma(\mathbb{Z}_{2^m}[i]))\backslash \{r, s, t, u\}$ is adjacent to $r$. Hence, the graph $\Gamma(\mathbb{Z}_{2^m}[i])$ is an $R$- graph of order greater than $5$. Thus, by Lemma $3.6$, $L(\Gamma(\mathbb{Z}_{2^m}[i]))$ is pancyclic.
	
\end{proof}

\begin{theorem}
	
	$L(\Gamma(\mathbb{Z}_{q^m}[i]))$ is pancyclic, for $m \geq 2$.	
	
\end{theorem}

\begin{proof}
	
	Since for the prime $q$, $\mathbb{Z}_{q}[i]$ is a field, so its zero divisor graph is an empty graph. For $m=2$, $\Gamma(\mathbb{Z}_{q^2}[i])$ is a complete graph, therefore by Proposotion $3.1$, $L(\Gamma(\mathbb{Z}_{q^2}[i]))$ is pancyclic. Now, for $m>2$, we consider four distinct vertices $r, s, t, u $ of $V(\Gamma(\mathbb{Z}_{q^m}[i]))$ such that $r = q^{m-1}$, $s = q^{\ceil*{\frac{m}{2} }}$, $t = iq^{\ceil*{\frac{m}{2} }}$ and $u = iq^{m-1}$. It is clear that $rs, st, tu, ur \in E(\Gamma(\mathbb{Z}_{q^m}[i]))$ and for every  $v \in V(\Gamma(\mathbb{Z}_{q^m}[i]))\backslash \{r, s, t, u\}$ is adjacent to $r$. Hence, the graph $\Gamma(\mathbb{Z}_{q^m}[i])$ is an $R$-graph of order greater than $5$. Thus, by Lemma $3.6$, $L(\Gamma(\mathbb{Z}_{q^m}[i]))$ is pancyclic.
	
\end{proof}

\begin{theorem}
	
	$L(\Gamma(\mathbb{Z}_{p^m}[i]))$ is pancyclic if and only if $m = 1$.
	
\end{theorem}

\begin{proof}
	
If $m = 1$, then $\Gamma(\mathbb{Z}_{p}[i])$ is complete bipartite graph $K_{p-1, p-1}$ and this implies it is a Hamiltonian. Now, taking $uv \in E(\Gamma(\mathbb{Z}_{p}[i]))$, $d(u) + d(v) = p-1 + p-1 = 2p-2 > \frac{(2n+1)}{3} $, where $n = 2p-2$, then $L(\Gamma(\mathbb{Z}_{p}[i]))$ is Hamiltonian and also $G \ncong C_4, C_5$, therefore by Proposition $3.1$, $L(\Gamma(\mathbb{Z}_{p}[i]))$ is pancyclic. If $m \neq 1$, then from $(Corollary ~4.4 )$ of \cite{naz}, $L(\Gamma(\mathbb{Z}_{p^m}[i]))$ is not Hamiltonian and so 	 $L(\Gamma(\mathbb{Z}_{p^m}[i]))$ is not pancyclic.
	
\end{proof}

\section{When is $\overline{\Gamma(\mathbb Z_n[i])}$ pancyclic$?$}

For $n = 2^m$, $\overline{\Gamma(\mathbb Z_{2^m}[i])}$ is not Hamiltonian because it contains isolated vertex $\langle (1+ i)^{2m-1}\rangle \backslash \{0\}$.

An important result given by G. H Fan in \cite{fan} will help to prove the Theorem $4.2$.

\begin{theorem}
	
\cite{fan}	Let $G$ be a $2$-connected graph on $n > 3$ vertices and $v$,  $u$ be distinct vertices of $G$. If
	\begin{center}
			
	$d(v, u) = 2 \Rightarrow max(d(v), d(u)) \geq n/2,$
\end{center}

then $G$ has a Hamiltonian cycle.

\end{theorem}

\begin{theorem}

For $m > 1$, $\overline{\Gamma(\mathbb Z_{p^m}[i])}$ is pancyclic graph.

\end{theorem}

\begin{proof}

For $m = 1$, $\overline{\Gamma(\mathbb Z_{p}[i])} \cong K_{p-1} \cup K_{p-1}$, which is disconnected. So $\overline{\Gamma(\mathbb Z_{p}[i])}$ is not a pancyclic. Now, for $m > 1$, $\overline{\Gamma(\mathbb Z_{p^m}[i])} \cong \overline{\Gamma(\mathbb Z_{p^m} \times \mathbb Z_{p^m})}$. Here degree of each vertex is greater than $n/2 $ where $n = \lvert \Gamma(\mathbb Z_{p^m}[i]) \rvert$ except $(0, sp^{n-1})$ and $(sp^{n-1}, 0)$, where $1 \leq s \leq p-1$. Since $d\{ (0, sp^{n-1}), (sp^{n-1}, 0)\} = 3$, because $(0, sp^{n-1})--(0, 1)--(1,p)--(sp^{n-1}, 0)$ or $(0, sp^{n-1})--(p, 1)--(1,0)--(sp^{n-1}, 0)$. Then for every $u, v \in \overline{\Gamma(\mathbb Z_{p^m}[i])}$, $d(u, v) = 2$, $max(d(u), d(v)) \geq n/2$ hold. Also, this is a $2-$connected, therefore by theorem $4.1$, $\overline{\Gamma(\mathbb Z_{p^m}[i])}$ is Hamiltonian and edges of $\overline{\Gamma(\mathbb Z_{p^m}[i])}$ $i.e$ $ \lvert E(G) \rvert \geq n^2/4$. Hence, by Bondy \cite{bondy}, $\overline{\Gamma(\mathbb Z_{p^m}[i])}$ is either pancyclic of complete bipartite graph. Since $(0,1)--(0,2)--(0,3)$ in $\overline{\Gamma(\mathbb Z_{p^m}[i])}$, so it contains a cycle of length $3$. Thus, $\overline{\Gamma(\mathbb Z_{p^m}[i])}$ is pancyclic graph.

\end{proof}

For $\overline{\Gamma(\mathbb Z_{q^m}[i])}$, here every vertex $\langle q\rangle$ is adjacent to $\langle q^{n-1}\rangle$ in $\Gamma(\mathbb Z_{q^m}[i])$. so $\langle q^{n-1}\rangle$ is an isolated vertex in $\overline{\Gamma(\mathbb Z_{q^m}[i])}$. Hence $\overline{\Gamma(\mathbb Z_{q^m}[i])}$ will never be pancyclic.

\section{When is $L(\overline{\Gamma(\mathbb Z_n[i])})$ pancyclic$?$}

Since $V(\Gamma(\mathbb{Z}_2 [i])) =\{1+i\}$, so its complement graph is $K_0$. $\Gamma (\mathbb{Z}_{q^2} [i])$ is a complete graph and $L(\overline{\Gamma(\mathbb Z_{q^2}[i])})$ is $K_0$. Now $\overline{\Gamma(\mathbb Z_{2^m}[i])},~ m\geq 2$ is a disconnected graph with two component. One is isolated vertex $\{2^{m-1} + i2^{m-1}\}$ and other is connected subgraph called $H$. So, $L(\overline{\Gamma(\mathbb Z_{2^m}[i])}) \cong L(H)$. Similarly, $\overline{\Gamma(\mathbb Z_{q^m}[i])}$ is also disconnected graph in which connected component is called $H$. Then $L(\overline{\Gamma(\mathbb Z_{q^m}[i])}) \cong L(H)$. Now, we prove following theorem.

\begin{theorem}
If $n = 2^m, ~ m\geq 2$, then $L(\overline{\Gamma(\mathbb Z_n[i])})$ is pancyclic.	
\end{theorem}

\begin{proof}
	
	Since$\Gamma(\mathbb Z_{2^m}[i]),$ has $2^{2m-2}$ vertices having degree $1$. In $\overline{\Gamma(\mathbb Z_{2^m}[i])}$, the connected component $H$ has complete subgraph of order $2^{2m-2}$ and other vertices of $H$ are adjacent to atleast $2^{2m-2}$ vertices. So, $\delta (H) \geq 2^{2m-2}$. Now taking for every $u, v \in H$, $deg(u) + deg(v) \geq 2^{2m-2} + 2^{2m-2} = 2^{2m-1} > \frac{2m+1}{3}$ and $H \neq C_4, C_5$. Then by proposition $3.1$, $L(H)$ is pancyclic and $L(\overline{\Gamma(\mathbb Z_{2^m}[i])}) \cong L(H)$. Hence, $L(\overline{\Gamma(\mathbb Z_{2^m}[i])})$ is pancyclic.
	
	\end{proof}

\begin{theorem}
	If $n = q^m, ~ m\geq 3$, then $L(\overline{\Gamma(\mathbb Z_n[i])})$ is pancyclic.
\end{theorem}

\begin{proof}
	In $\overline{\Gamma(\mathbb Z_{q^m}[i])}$, the connected component is assumed to be $H$. Taking four distinct vertices $r, s, t, u $ of $V(H)$ where $r = q$, $s = q^2+iq$, $t =q + iq^2$, $u = iq$ such that $rs, st, tu, ur$ have an edge in $E(H)$ and for every $v \in V(H)\backslash \{r, s, t, u\}$ is adjacent to $r$. So the graph $H$ is an $R-$graph. Thus by Lemma $3.7$, $L(H)$ is pancyclic and hence $L(\overline{\Gamma(\mathbb Z_{q^m}[i])})$ is pancyclic.

\end{proof}

\bibliographystyle{amsplain}

\end{document}